\documentclass{amsart}
\usepackage{amsthm,amscd,amssymb,verbatim,epsf,amsmath,amsfonts,mathrsfs,graphicx}
\usepackage[colorlinks=true,linkcolor=blue,citecolor=blue]{hyperref}

\begin{document}
\theoremstyle{plain}
\newtheorem{Thm}{Theorem}
\newtheorem{Cor}{Corollary}
\newtheorem{Ex}{Example}
\newtheorem{Con}{Conjecture}
\newtheorem{Main}{Main Theorem}
\newtheorem{Lem}{Lemma}
\newtheorem{Prop}{Proposition}

\theoremstyle{definition}
\newtheorem{Def}{Definition}
\newtheorem{Note}{Note}

\theoremstyle{remark}
\newtheorem{notation}{Notation}
\renewcommand{\thenotation}{}
\setcounter{tocdepth}{1}
\errorcontextlines=0
\numberwithin{equation}{section}
\renewcommand{\rm}{\normalshape}%

\title[Neutral 4-manifolds with null boundary]%
   {The causal topology of neutral 4-manifolds \\with null boundary}

\author{Nikos Georgiou}
\address{Nikos Georgiou\\
          Department of Mathematics\\
          Waterford Institute of Technology\\
          Waterford\\
          Co. Waterford\\
          Ireland.}
\email{ngeorgiou@wit.ie}
\author{Brendan Guilfoyle}
\address{Brendan Guilfoyle\\
          School of Science, Technology, Engineering and Mathematics,\\
          Institute of Technology, Tralee\\
          Clash\\
          Tralee\\
          Co. Kerry\\
          Ireland.}
\email{brendan.guilfoyle@ittralee.ie}

\keywords{Neutral metric, null boundary, hyperbolic 3-space, 3-sphere, spaces of constant curvature, geodesic spaces, contact}
\subjclass{Primary: 53A35; Secondary: 57N13}

\begin{abstract}
This paper considers aspects of 4-manifold topology from the point of view of the null cone of a neutral metric, a point of view we call neutral 
causal topology. In particular, we construct and investigate neutral 4-manifolds with null boundaries that arise from canonical 3- and 4-dimensional 
settings. 

A null hypersurface is foliated by its normal and, in the neutral case, inherits a pair of totally null planes at each point. This paper focuses on these
plane bundles in a number of classical settings
 
The first construction is the conformal compactification of flat neutral 4-space into the 4-ball. The null foliation on the boundary in this case is the 
Hopf fibration on the 3-sphere and the totally null planes in the boundary are integrable. The metric on the 4-ball is a conformally flat, scalar-flat, 
positive Ricci curvature neutral metric. 

The second constructions are subsets of the 4-dimensional space of oriented geodesics in a 3-dimensional space-form, equipped with its canonical
neutral metric. We consider all oriented geodesics tangent to a given embedded strictly convex 2-sphere. Both totally null planes on this null 
hypersurface are contact, and we characterize the curves in the null boundary that are Legendrian with respect to either totally null plane bundles. 
The Reeb vector 
field associated with the alpha-planes are shown to be the oriented normal lines to geodesics in the surface.

The third is a neutral geometric model for the intersection of two surfaces in a 4-manifold.
The surfaces are the sets of oriented normal lines to two round spheres in Euclidean 3-space, which form Lagrangian surfaces in the 4-dimensional 
space of all oriented lines. The intersection of the boundaries of their normal neighbourhoods 
form tori that we prove are totally real and Lorentz if the spheres do not intersect.  

We conclude with possible topological applications of the three constructions, including neutral Kirby calculus, neutral knot invariants
and neutral Casson handles, respectively. 

\end{abstract}

\date{23rd October 2016}
\thanks{Expository video clips explaining the results and motivations of this paper can be found at the following link: \url{https://www.youtube.com/watch?v=VUlPMPwT-hA}}
\maketitle
\newpage
\tableofcontents

\section{Introduction}

This paper considers certain 4-manifolds with boundary which carry a neutral metric (pseudo-Riemannian of signature (2,2))
with respect to which the boundary is a null hypersurface. We seek to extract geometric and topological information from the null cone of 
such metrics in a number of canonical situations. 

The results can be viewed as the first steps in the development of a neutral causal topology for 4-manifolds with boundary. From this point of view,
section \ref{s2} presents the 0-handle of a neutral Kirby calculus, with preferred curves along which to do surgery. The neutral metric appears to be 
ideally suited to 2-handle constructions in which the framing is tracked by the null cone on the associated tori.

Section \ref{s3} develops the theory of knots in tangent hypersurfaces in order to identify neutral knot invariants in null boundaries, while Section
\ref{s4} constructs a local geometric model for the normal neighbourhood of a transverse double point of a Lagrangian disc.

In more detail, we consider the conformal compactification of an open neutral 4-manifold.
Conformal compactifications of both Riemannian and Lorentzian 4-manifolds have been 
long studied \cite{and01} \cite{PaR}. For neutral 4-manifolds even the flat case has not received much attention. 
In the next section we seek to remedy this by providing the canonical example: 

\vspace{0.1in}
\begin{Thm}\label{t1}
There exists a smooth embedding $f:({\mathbb R}^{2,2},{\mathbb G})\rightarrow(B^4,\tilde{{\mathbb G}})$ and a function 
$\Omega:B^4\rightarrow{\mathbb R}$: such that
\begin{itemize}
\item[(i)] $f$ is a conformal diffeomorphism onto the interior of $B^4$ with  $f^*{\tilde {\mathbb G}}=\Omega^2{\mathbb G}$, 
\item[(ii)] $\Omega=0$ on $\partial B^4=S^3$,
\item[(iii)] the boundary is null,
\item[(iv)] $d\Omega=0$ on the boundary $S^3$ precisely on an embedded Hopf link.
\end{itemize}
\end{Thm}
\vspace{0.1in}

The metric $\tilde{\mathbb G}$ on the 4-ball is a conformally flat, scalar-flat neutral metric with positive definite Ricci tensor, 
analogous to the Einstein static universe.  Thus, space-like infinity and timelike infinity are Hopf-linked in the 
boundary of a flat universe with two times.

The null boundary inherits a degenerate Lorentz metric, whose null cone is a pair of transverse totally null planes at each point 
($\alpha$-planes and $\beta$-planes).  In the conformal compactification of ${\mathbb R}^{2,2}$ these plane fields are both integrable, and contain the
tangents to the (1,1) and (1,-1) curves on the Hopf tori about the link. 

This 4-ball should be viewed as the 0-handle of a neutral Kirby calculus so that one can consider attaching handles along framed curves
in the boundary \cite{gompf} \cite{kirby}. In order to carry the neutral metric along certain causal conditions must be fulfilled and then one can
develop a neutral surgery on conformal classes of neutral metrics. In this case, the foliation by Lorentz tori tracks the framing for such surgery
along the Hopf link.

The second type of 4-manifold, detailed in Section \ref{s3}, are subsets of the space ${\mathbb L}({\mathbb M}^3)$ of oriented geodesics in a 
3-dimensional space-form $({\mathbb M}^3,g)$.
It is well-known that ${\mathbb L}({\mathbb M}^3)$ admits an invariant neutral metric ${\mathbb G}$  \cite{gag1} \cite{gk1} \cite{honda}
\cite{sal05} \cite{Salvai}. 

Given a smoothly embedded surface $S\subset{\mathbb M}^3$, define the {\it tangent hypersurface of S}, denoted 
${\mathcal{H}}(S)\subset {\mathbb L}({\mathbb M}^3)$, to be the set of oriented 
geodesics that are tangent to $S$. This 3-manifold is locally a circle bundle over $S$, with projection $\pi:{\mathcal{H}}(S)\rightarrow S$ and fibre 
generated by rotation about the normal to $S$. 

In this paper we investigate the geometric properties of ${\mathcal H}(S)$ induced by the neutral metric on ${\mathbb L}({\mathbb M}^3)$.
If $S\subset{\mathbb M}^3$ is a smooth surface, then  ${\mathcal H}(S)$ is an immersed hypersurface which is null with respect to 
${\mathbb G}$. 

Thus, ${\mathcal H}(S)$ is foliated by null geodesics and contains an $\alpha$-plane and a $\beta$-plane at each point. A knot 
${\mathcal C}\subset{\mathcal H}(S)$, which is an oriented tangent line field over a curve $c\subset S$, is said to be {\it$\alpha$-Legendrian} 
({\it$\beta$-Legendrian}) if its tangent lies in the $\alpha$-planes ($\beta$-planes, respectively).

Given a contact structure on a 3-manifold with contact 1-form $\omega$, the {\it Reeb vector field} $X$ is characterised by
\[
d\omega(X,\cdot)=0 \qquad\qquad \omega(X)=1
\]

In the case where $S$ is a strictly convex 2-sphere, the tangent hypersurface bounds a disc bundle of
Euler number 2 in ${\mathbb L}({\mathbb M}^3)$, and we prove:

\vspace{0.1in}
\begin{Thm}\label{t2}
If $S\subset{\mathbb M}^3$ is a smooth convex 2-sphere, then the $\alpha$-planes and $\beta$-planes of the neutral metric are both contact.

Moreover, a knot ${\mathcal C}\subset{\mathcal H}(S)$, with contact curve $c=\pi({\mathcal C})\subset S$, is $\alpha$-Legendrian iff  
$\forall \gamma\in{\mathcal C}$, $\gamma$ is tangent to $c\subset S\subset {\mathbb M}^3$.

In addition, any two of the following imply the third:
\begin{itemize}
\item[(i)] ${\mathcal C}$ is $\beta$-Legendrian, 
\item[(ii)] $\forall \gamma\in{\mathcal C}$, $\gamma$ is normal to $c$,
\item[(iii)] either $c$ is a line of curvature of $S$, or $S$ is umbilic along $c$.
\end{itemize}
Finally, the Reeb vector field of the $\alpha-$planes consists of the oriented lines normal to a geodesic of $S$.

\end{Thm}
\vspace{0.1in}
The proof requires separate formalisms in the flat and non-flat cases.

Section \ref{s4} contains a local geometric model of the normal neighbourhood of an isolated double point on an immersed surface, given by
the intersection of two Lagrangian surfaces in ${\mathbb L}({\mathbb R}^3)$. These surfaces are the oriented normal lines to two round spheres
in ${\mathbb R}^3$ and the boundaries of a normal neighbourhood of the surfaces can be identified with the tangent hypersurfaces of the spheres.

\vspace{0.1in}
\begin{Thm}\label{t3}
Let $S_1,S_2\subset{\mathbb R}^3$ be round spheres of radii $r_1\geq r_2$ with centres separated by a distance $l$ in ${\mathbb R}^3$.

Then,
\begin{itemize}
\item[(i)]  $\qquad{\mathcal H}(S_1)\cap{\mathcal H}(S_2)=\{\emptyset\}\qquad{\mbox{ iff }}\qquad l<r_1-r_2$, 
\item[(ii)] $\qquad{\mathcal H}(S_1)\cap{\mathcal H}(S_2)=S^1\qquad{\mbox{ iff }}\qquad l=r_1-r_2$,
\item[(iii)] $\qquad{\mathcal H}(S_1)\cap{\mathcal H}(S_2)=T^2\qquad{\mbox{ iff }}\qquad r_1-r_2<l\leq r_1+r_2$,
\item[(iv)] ${\mathcal H}(S_1)\cap{\mathcal H}(S_2)=T^2\coprod T^2\qquad{\mbox{ iff }}\qquad r_1+r_2<l$.
\end{itemize}
 
If $l>r_1+r_2$ so that $S_1\cap S_2=\{\emptyset\}$, then the intersection tori $T^2$ are totally real and Lorentz.
\end{Thm}
\vspace{0.1in}

In the final Section, we discuss these three constructions from a topological point of view.

\vspace{0.2in}

\section{Conformal Compactification}\label{s2}

\subsection{Neutral geometry}
Let us assemble some facts of neutral geometry that will be required in this paper. The statements are in ${\mathbb R}^4$, but hold 
in the tangent space at a point in any neutral 4-manifold.

Consider the flat neutral metric ${\mathbb G}$ on ${\mathbb R}^4$ in standard coordinates $(x^1,x^2,x^3,x^4)$:
\[
ds^2=(dx^1)^2+(dx^2)^2-(dx^3)^2-(dx^4)^2.
\]
Throughout, denote ${\mathbb R}^4$ endowed with this metric by ${\mathbb R}^{2,2}$.

\vspace{0.1in}
\begin{Def}
The neutral {\it null cone} is the set of null vectors in ${\mathbb R}^{2,2}$:
\[
{\mathcal K}=\{X\in{\mathbb R}^{2,2}\;|\; {\mathbb G}(X,X)=0\}.
\]
\end{Def}
\vspace{0.1in}

The null cone is a cone over a torus, in distinction to the Lorentz ${\mathbb R}^{3,1}$ case where the null cone is a cone over a 2-sphere. To
see the torus, note that the map $f:{\mathbb R}\times S^1\times S^1\rightarrow {\mathcal K}$ 
\[
f(a,\theta_1,\theta_2)=\left( a\cos\theta_1,a\sin\theta_1,a\cos\theta_2,a\sin\theta_2\right),
\]
parameterizes the null vectors as a cone over ${\mathbb T}^2$.

\begin{Def}
A plane $P\subset{\mathbb R}^{2,2}$ is {\it totally null} if every vector in $P$ is null with respect to ${\mathbb G}$, and the inner 
product of any two vectors in $P$ is zero. 
\end{Def}

Since every vector that lies in a totally null plane is null, we can picture a null plane as a cone over a circle in ${\mathcal K}$. A straight-forward
calculation shows that:

\begin{Prop}
A totally null plane is a cone over either a (1,1)-curve or a (1,-1)-curve on the torus, the former for an $\alpha$-plane, the latter for a
$\beta$-plane. 
\end{Prop}
\vspace{0.1in}

By rotating around the meridian we see that the set of totally null planes is $S^1\coprod S^1$.

The metric has two natural compatible complex structures (up to an overall sign), which in coordinates $(x^1,x^2,x^3,x^4)$ take the form
\[
{\mathbb J}^+=\left[\begin{matrix}
0 & 1&0&0 \\
-1 & 0&0&0\\
0 & 0&0&1\\
0 & 0&-1&0
\end{matrix}
\right]
\qquad\qquad
{\mathbb J}^-=\left[\begin{matrix}
0 & 1&0&0 \\
-1 & 0&0&0\\
0 & 0&0&-1\\
0 & 0&1&0
\end{matrix}
\right].
\] 

\vspace{0.1in}
\begin{Prop}\cite{gag2}
An $\alpha-$plane ($\beta-$plane) is invariant under the complex structure ${\mathbb J}^+$ ( ${\mathbb J}^-$), respectively
\end{Prop}
\vspace{0.1in}

Note that, in general, these only extend to compatible almost complex structures on a neutral manifold. The space of 
compatible almost complex structures on a neutral 4-manifold is referred to as the hyperbolic twistor space of the metric \cite{LaM}. 

Composition of either of the complex structures with the metric yields a 2-form, which is symplectic in the flat case. However, the 2-form does not 
tame the almost complex struture in the sense of Gromov \cite{gromov} - neutral metrics walk on the wild side.

Now consider a null vector $X\in{\mathbb R}^{2,2}$. The set of vectors orthogonal to $X$ is 3-dimensional and contains the vector $X$ itself. Choosing 
another null vector $Y$ which has ${\mathbb G}(X,Y)=1$, complete this to a frame $\{{e}_{+},{e}_{-},{e}_{0}=X,{f}_{0}=Y\}$ such that 
\[
{\mathbb G}=\left[
\begin{array}{ccc|c}
1 & 0 & 0 & 0\\
0 & -1 & 0 & 0\\
0 & 0 & 0 & 1 \\\hline
0 & 0 & 1 & 0 
\end{array}
\right].
\]
Clearly the hypersurface orthogonal to $X$ has a degenerate Lorentz metric and the set of null vectors at each point consists of two totally null planes, 
intersecting along the normal vector $X$. This structure exists on any null hypersurface in any neutral 4-manifold and will be considered in some detail 
in the constructions of this paper.

\vspace{0.2in}
\subsection{The conformal compactification of ${\mathbb R}^{2,2}$}

We will now conformally embed ${\mathbb R}^{2,2}$ as an open 4-ball in ${\mathbb R}^4$ so that the points at infinity in ${\mathbb R}^{2,2}$
form the boundary 3-sphere.

First, let us introduce the coordinate change $(x^1,x^2,x^3,x^4)\rightarrow(R_1,R_2,\theta_1,\theta_2)$ defined by the double polar transformation:
\begin{equation}\label{e:coordsfl}
x^1+ix^2=R_1e^{i\theta_1} \qquad\qquad x^3+ix^4=R_2e^{i\theta_2}.
\end{equation}
To bring the points at infinity (i.e. $R_1$ or $R_2$ going to infinity) in to a finite distance define
\[
\tan p=R_1+R_2 \qquad\qquad \tan q=R_1-R_2.
\]
Clearly the coordinates $(p,q,\theta_1,\theta_2)$, with 
\[
0\leq p <\pi/2 \qquad\qquad -p\leq q\leq p \qquad 0\leq \theta_1,\theta_2<2\pi,
\]
cover all of ${\mathbb R}^{2,2}$. Moreover, infinity has been brought in to the boundary $p=\pi/2$.

This boundary is in fact a 3-sphere bounding a 4-ball $B^4$, as can be seen by the identification of $(z_1,z_2)\in{\mathbb C}^2={\mathbb R}^4$ 
\begin{equation}\label{e:coordscomp}
z_1=p\sin(\psi/2)e^{i\theta_1} \qquad \qquad z_2=p\cos(\psi/2)e^{i\theta_2}, 
\end{equation}
where $q=p\cos\psi$ with $0\leq\psi\leq\pi$. The boundary is the 3-sphere $\partial B^4=S^3$ of radius $\pi/2$ and the tori parameterized by 
$(\theta_1,\theta_2)$ are exactly the hopf tori in $S^3$.

Under the diffeomorphism $(x^1,x^2,x^3,x^4)\rightarrow(p,q,\theta_1,\theta_2)$ the pull-back of the metric $\tilde{{\mathbb G}}$ is the conformal metric 
$\tilde{{\mathbb G}}=\Omega^2{\mathbb G}=4\cos^2p\cos^2q \;{\mathbb G}$
given by
\begin{equation}\label{e:confmet}
d\tilde{s}^2=dpdq+{\textstyle{\frac{1}{4}}}\sin^2(p+q)d\theta_1^2-{\textstyle{\frac{1}{4}}}\sin^2(p-q)d\theta_2^2,
\end{equation} 
which is a conformally flat, scalar-flat neutral metric, analogous to the Einstein static universe. The Ricci tensor 
of $\tilde{{\mathbb G}}$ has non-vanishing components:
\[
\tilde{R}_{pp}=\bar{R}_{qq}=2 \qquad\qquad \tilde{R}_{\theta_1\theta_1}=\sin^2(p+q) \qquad\qquad \tilde{R}_{\theta_2\theta_2}=\sin^2(p-q).
\]

Clearly, the boundary 3-sphere is null and this has interesting consequences. The normal vector lies in the sphere. In general the set of points on the 
boundary at which $d\Omega$ vanishes 
would be zero dimensional, the fact that the hypersurface is null (so that $|d\Omega|=0$ {\it everywhere} on the boundary) means that the zero locus 
is 1-dimensional. 

A short calculation shows that $d\Omega=0$ on $S^3$ when $q=\pm{\textstyle{\frac{\pi}{2}}}$. Since $p={\textstyle{\frac{\pi}{2}}}$, we have
$\psi\in\{0,\pi\}$ and equations (\ref{e:coordscomp}) tell us that the gradient of the conformal factor vanishes on a pair of Hopf-linked circles in the 
boundary.

We have now proven Theorem \ref{t1} and propose that the four conditions of this Theorem are natural for the conformal compactification 
of more general neutral 4-manifolds - with the Hopf link replaced by some other link in the boundary.

The metric induced on a null hypersurface by a neutral metric has degenerate signature $(0,+,-)$ and the null cone degenerates to a pair of 
totally null planes, called $\alpha-$planes and $\beta-$planes, which intersect on the normal to the hypersurface, which, being null, lies in the 
tangent space to the hypersurface. 

\vspace{0.1in}
\begin{Prop}
Both the $\alpha-$planes and $\beta-$planes on the boundary are integrable. 
\end{Prop}
\begin{proof}
The pullback of the metric (\ref{e:confmet}) onto the boundary 3-sphere $p=\textstyle{\frac{\pi}{2}}$ is
\[
d\tilde{s}^2|_{{\mathbb S}^3}={\textstyle{\frac{1}{4}}}\cos^2q\left(d\theta_1^2-d\theta_2^2\right),
\]
and so the null cone is spanned by
\[
X_{\pm}=a\frac{\partial}{\partial q}+b\left(\frac{\partial}{\partial \theta_1}\pm\frac{\partial}{\partial \theta_2}\right).
\]
The 1-forms that vanish on these two planes are proportional to
\[
\omega_{\pm}=d\theta_1\mp d\theta_2,
\]
so that $\omega_{\pm}\wedge d\omega_{\pm}=0$ and the distributions are integrable.
\end{proof}

\vspace{0.1in}
Note here that the null planes intersect the tori $q=constant$ in the (1,1) and (1,-1) curves, which gives the null cone structure 
on these Lorentz tori.

\vspace{0.1in}

The existence of a conformal compactification with null boundary means that the metric ${\mathbb G}$ must be scalar flat at
infinity in the original 4-manifold, since by the well-known conformal change
\[
\Omega^2\bar{R}=R-6\Omega \bar{\triangle}\Omega+12|\bar{\nabla}\Omega|^2
\] 
along the null boundary  $|\bar{\nabla}\Omega|^2=0$ and so $R\rightarrow 0$ as $\Omega\rightarrow 0$. In the 4-manifolds we consider, it
is scalar flat throughout and so this obstruction does not arise.

\newpage
\section{Tangent Hypersurfaces}\label{s3}

\subsection{Flat 3-space}

\subsubsection{The neutral metric}

Interest in the neutral metric on the space of oriented geodesics of a 3-dimensional space of constant curvature has grown recently 
\cite{gag1} \cite{gk1} \cite{honda} \cite{sal05} \cite{Salvai}. The underlying smooth 4-manifold in the ${\mathbb R}^3$ case 
is the total space of the tangent bundle to the 2-sphere ${\mathbb L}({\mathbb R}^3)\equiv T{\mathbb S}^2$, and we adopt the notation of \cite{gk2} 
for the local description.

This identification is made concrete by choosing Euclidean coordinates $(x^1,x^2,x^3)$ and considering tangent vectors to the unit 2-sphere in 
the same ${\mathbb R}^3$. Thus, choosing holomorphic coordinates about the north pole on ${\mathbb S}^2$, the tangent vector
\[
V=\eta\frac{\partial}{\partial \xi}+ \bar{\eta}\frac{\partial}{\partial \bar{\xi}},
\]
for $\eta\in{\mathbb C}$ is identified with the oriented parameterized line $\gamma:{\mathbb R}\rightarrow{\mathbb R}^3:r\mapsto\gamma(r)$ given by
\begin{equation}\label{e:1a}
z=x^1+ix^2=\frac{2(\eta-\bar{\xi}^2\bar{\eta})}{(1+\xi\bar{\xi})^2}+\frac{2\xi}{1+\xi\bar{\xi}}r,
\end{equation}
\begin{equation}\label{e:1b}
x^3=-\frac{2(\xi\bar{\eta}+\bar{\xi}\eta)}{(1+\xi\bar{\xi})^2}+\frac{1-\xi\bar{\xi}}{1+\xi\bar{\xi}}r.
\end{equation}

In this transform, $(\xi,\eta)$ are local holomorphic coordinates for ${\mathbb L}({\mathbb R}^3)$ with the fibre over the south pole removed. In fact
${\mathbb L}({\mathbb R}^3)$ admits a pair of canonical complex structures ${\mathbb J}^+$ and ${\mathbb J}^-$  which when expressed in the 
coordinates $(\xi,\bar{\xi},\eta,\bar{\eta})$ take the form
\[
{\mathbb J}^+=\left[\begin{matrix}
i & 0&0&0 \\
0 & -i&0&0\\
0 & 0&i&0\\
0 & 0&0&-i
\end{matrix}
\right]
\qquad\qquad
{\mathbb J}^-=\left[\begin{matrix}
0 & 0&1&0 \\
0 & 0&0&-1\\
-1 & 0&0&0\\
0 & 1&0&0
\end{matrix}
\right].
\] 
In addition, there is a metric ${\mathbb G}$ on ${\mathbb L}({\mathbb R}^3)$ that is invariant under the Euclidean group, which takes the form
\begin{equation}\label{e:metric}
{\Bbb{G}}=2(1+\xi\bar{\xi})^{-2}{\Bbb{I}}\mbox{m}\left(d\bar{\eta} d\xi+\frac{2\bar{\xi}\eta}{1+\xi\bar{\xi}}d\xi d\bar{\xi}\right).
\end{equation}
Clearly, the metric is compatible with ${\mathbb J}^+$, but not with ${\mathbb J}^-$. The complex structure ${\mathbb J}^+$ has played a
significant role in holomorphic methods applied to Euclidean problems, such as monopoles \cite{hitch} and minimal surfaces \cite{weier}. 

The composition of ${\mathbb J}^+$ and the neutral metric ${\mathbb G}$ yields a symplectic form
\begin{equation}\label{e:symplectic}
\Omega=2(1+\xi\bar{\xi})^{-2}{\Bbb{R}}\mbox{e}\left(d\bar{\eta}\wedge d\xi+\frac{2\bar{\xi}\eta}{1+\xi\bar{\xi}}d\xi\wedge d\bar{\xi}\right).
\end{equation}

While this symplectic structure does not tame ${\mathbb J}^+$, it has the following property: a surface $\Sigma$ in ${\mathbb L}({\mathbb R}^3)$, 
that is, a 2-parameter family of oriented lines, is normal to a surface in ${\mathbb R}^3$ iff $\Sigma$ is Lagrangian: $\Omega_\Sigma=0$.

\vspace{0.2in}
\subsubsection{Tangent hypersurfaces}

For any smoothly embedded convex surface $S\subset{\mathbb R}^3$ define the {\it tangent hypersurface} 
${\mathcal H}(S)\subset{\mathbb L}({\mathbb R}^3)$ to be
\[
{\mathcal H}(S)=\{\gamma\in{\mathbb L}({\mathbb M}^3) \;|\;\gamma\in T_{\gamma\cap S}S\;\}.
\]
Clearly rotation about the normal to $S$ at a point $p$ generates a circle in ${\mathcal H}(S)$, so that the hypersurface is a circle bundle over $S$.

From now on we assume that $S$ is a closed strictly convex surface, so that ${\mathcal H}(S)$ is an embedded copy of the unit tangent bundle to $S$.
Moreover 

\vspace{0.1in}
\begin{Prop}
The hypersurface ${\mathcal H}(S)$ is null with respect to ${\mathbb G}$ and foliated by null geodesic circles.
\end{Prop}
\begin{proof}

Rotating an oriented line about a line in ${\mathbb R}^3$ generates a null geodesic circle in ${\mathbb L}({\mathbb R}^3)$ \cite{gk1}. The 
tangent to these circles are in fact normal to ${\mathcal H}(S)$, as can be seen as follows.

Since $S\subset{\mathbb R}^3$ is convex it can be parameterized by the direction of its normal line. In local coordinates we have 
${\mathbb C}\rightarrow {\mathbb L}({\mathbb R}^3):\nu\mapsto (\xi=\nu,\eta=\eta_0(\nu,\bar{\nu}))$. It is well known that this is a Lagrangian section
of the canonical bundle $\pi:{\mathbb L}({\mathbb R}^3)\rightarrow{\mathbb S}^2$.

The point along the normal line where it intersects $S$ is determined by the support function $r_0:S\rightarrow{\mathbb R}$ which satisfies
\begin{equation}\label{e:pot}
\partial_\nu r_0=\frac{2\bar{\eta_0}}{1+\nu\bar{\nu}}.
\end{equation}
The sum and difference of the radii of curvature $r_1\geq r_2$ of $S$ are
\[
r_1+r_2=\psi_0
\qquad\qquad\qquad
r_1-r_2=|\sigma_0|,
\]
where
\begin{equation}\label{d:psi0s0}
\psi_0=r_0+2(1+\nu\bar{\nu})^2{\mathbb R}{\mbox{e }}\partial_\nu\left[\frac{\eta_0}{(1+\nu\bar{\nu})^2}\right]
\qquad\qquad\qquad
\sigma_0=-\partial_\nu\bar{\eta_0}.
\end{equation}

We are interested in the oriented lines that are tangent to $S$, that is, they are orthogonal to the normal.
\begin{Lem}
The oriented great circle in $S^2$ which is dual to the point with holomorphic coordinate $\nu$ is generated by
\begin{equation}\label{e:2}
\xi=\frac{\nu+e^{iA}}{1-\bar{\nu}e^{iA}},
\end{equation}
for $A\in [0,2\pi)$.
\end{Lem}

An oriented line $(\xi,\eta)$ passes through a point $(x^1,x^2,x^3)\in{\mathbb R}^3$ iff
\begin{equation}\label{e:3}
\eta={\textstyle{\frac{1}{2}}}\left(x^1+ix^2-2x^3\xi-(x^1-ix^2)\xi^2\right).
\end{equation}

Substituting equations (\ref{e:1a}), (\ref{e:1b}) with $(\xi,\eta)=(\nu,\eta_0)$ and $r=r_0$, and (\ref{e:2}) into (\ref{e:3}) yields
\begin{equation}\label{e:4a}
\eta=\frac{\eta_0-e^{2iA}\bar{\eta}_0-(1+\nu\bar{\nu})e^{iA}r_0}{(1-\bar{\nu} e^{iA})^2}.
\end{equation}

Thus, the hypersurface ${\mathcal H}(S)$ is locally parameterized by (\ref{e:2}) and (\ref{e:4a}) for $(\nu,\bar{\nu})$ varying over the normal 
directions of $S$ and $A\in S^1$. 

Pulling back the metric onto ${\mathcal H}$, we find that the induced metric in these coordinates (making use of equation (\ref{e:pot}) and definitions 
(\ref{d:psi0s0})) is
\begin{equation}\label{e:indmet}
ds^2=-\frac{2}{(1+\nu\bar{\nu})^2}{\mathbb I}{\mbox{m}}\left[(\sigma_0+\psi_0e^{-2iA})d\nu^2+\sigma_0e^{2iA}d\nu d\bar{\nu}\right].
\end{equation}
Thus the metric is degenerate along the null vector in the $A$-direction.  This completes the proof.
\end{proof}
\vspace{0.1in}

The null vectors tangent to ${\mathcal H}(S)$ form a pair of planes, the $\alpha-$planes and $\beta-$planes, which intersect on the null normal. 
The former planes are preserved  by the complex structure ${\mathbb J}^+$ and the latter by ${\mathbb J}^-$ \cite{gag2}.

\vspace{0.1in}
\begin{Prop}\label{p:con_flat}
If $S\subset{\mathbb R}^3$ is a smooth convex 2-sphere, then the $\alpha$-planes and $\beta$-planes of the neutral metric are both contact.
\end{Prop}
\begin{proof}
Consider the induced metric (\ref{e:indmet}) and write down the null planes. In particular

\vspace{0.1in}
\begin{Lem}\label{l:albe}
The vector $\vec{X}\in T_{(\nu,A)}{\mathcal H}(S)$
\[
\vec{X}=a\frac{\partial}{\partial A}+b{\mathbb R}{\mbox{e}}\left[e^{iB}\frac{\partial}{\partial \nu}\right],
\]
for $a,b\in{\mathbb R}$, is null iff either
\[
B=A+\frac{1}{2i}\ln\left(\frac{\psi_0+\bar{\sigma}_0e^{-2iA}}{\psi_0+{\sigma}_0e^{2iA}}\right)
\qquad\qquad{\mbox or }\qquad\qquad
B=A+\frac{\pi}{2}.
\]
The former spans the $\alpha-$plane, while the latter the $\beta-$plane.
\end{Lem}
\vspace{0.1in}

The 1-form $\omega^+$ that vanishes on the $\alpha-$plane is
\begin{equation}\label{e:alcontf}
\omega^+=-2{\mathbb I}m\frac{e^{-iA}\psi_0+e^{iA}\sigma_0}{1+\nu\bar{\nu}}d\nu,
\end{equation}
and so
\[
\omega^+\wedge d\omega^+=-\frac{2i(\psi_0^2-\sigma_0\bar{\sigma}_0)}{(1+\nu\bar{\nu})^2}dA\wedge d\nu\wedge d\bar{\nu}.
\]
For a convex surface $\psi_0^2-\sigma_0\bar{\sigma}_0$ is never zero and so the distribution of $\alpha-$plane is contact.

On the other hand, the 1-form $\omega^-$ that vanishes on the $\beta-$plane is
\[
\omega^-=2{\mathbb R}e\;e^{-iA}d\nu,
\]
and so
\[
\omega^-\wedge d\omega^-=-2idA\wedge d\nu\wedge d\bar{\nu}.
\]
Thus the distribution of $\beta-$plane is contact.
\end{proof}
\vspace{0.1in}

Note that these hypersurfaces sit within a wider class of oriented lines passing through $S$ making an angle $0\leq a\leq\pi/2$ with the 
outward pointing normal:
\[
{\mathcal H}_a(S)=\{\gamma\in{\mathbb L}({\mathbb R}^3)\;|\;\gamma\cap S\neq\emptyset,\;  
  <\dot{\gamma},\hat{N}>=\cos a\;\},
\]
where $\dot{\gamma}$ is the direction of the oriented line $\gamma$ and $\hat{N}$ is the unit outward pointing normal vector. 

For $a=0$ this hypersurface degenerates to a Lagrangian surface in ${\mathbb L}({\mathbb R}^3)$, while for $a=\pi/2$ it is the tangent hypersurface. 
We refer to ${\mathcal H}_a(S)$ in the general $a>0$ case as the {\it constant angle hypersurface to S} which were first introduced in \cite{gk2}
while constructing a mod 2 neutral knot invariant.

The local equations for the ${\mathcal H}_a(S)$ (generalizing equations (\ref{e:2}) and (\ref{e:4a})) are
\begin{equation}\label{e:epstan}
\xi=\frac{\nu+\epsilon e^{iA}}{1-\bar{\nu}\epsilon e^{iA}}
 \qquad\qquad
\eta=\frac{\eta_0-\epsilon^2e^{2iA}\bar{\eta}_0-(1+\nu\bar{\nu})\epsilon e^{iA}r_0}{(1-\bar{\nu} \epsilon e^{iA})^2}
\end{equation}
where $\epsilon=\tan(a/2)$.

While these hypersurfaces are not null, they have the following property:

\vspace{0.1in}
\begin{Prop}
The hypersurface ${\mathcal H}_a(S)$ is null exactly at the oriented lines through an umbilic point on $S$ and at the oriented lines whose projection 
orthogonal to the normal is tangent to the lines of curvature of $S$.
\end{Prop}
\begin{proof}
This follows from pulling back the neutral metric (\ref{e:metric}) to the hypersurface (\ref{e:epstan}) and taking the determinant. The result is
\[
{\mbox{det }}{\Bbb{G}}|_{{\mathcal H}_a(S)}=-\frac{2\epsilon^2(1-\epsilon^2)^2
(\sigma_0 e^{2iA}-\bar{\sigma}_0e^{-2iA})(\psi_0^2-\sigma_0\bar{\sigma}_0)}{(1+\epsilon^2)^4(1+\nu\bar{\nu})^4},
\]
and the result follows.
\end{proof}

\vspace{0.1in}
We return to these hypersurfaces in Section \ref{s4} when considering normal neighbourhoods of Lagrangian discs in ${\mathbb L}({\mathbb R}^3)$.

Given the two contact distributions, introduce the following terminology:
\vspace{0.1in}
\begin{Def}
A knot ${\mathcal C}\subset{\mathcal H}(S)$ is {\it $\alpha$-Legendrian} ({\it $\beta$-Legendrian}) if its tangent lies in an $\alpha-$plane 
($\beta-$plane) at each point. 

The {\it contact curve} of ${\mathcal C}$ is the curve $c=\pi({\mathcal C})\subset S$ obtained by the canonical projection 
$\pi:{\mathcal H}(S)\rightarrow S$. 
\end{Def}
\vspace{0.1in}

The first part of Theorem \ref{t2} has been established by Proposition \ref{p:con_flat}, and we now prove the second part.

Let $c\subset S$ be a curve on a convex surface parameterized by arc-length $u\mapsto(x^1(u),x^2(u),x^3(u))$. Let $(\nu,\eta_0)$ be the 
outward pointing normal line to $S$ along $c$ so that 
\begin{equation}\label{e:2a}
z=x^1+ix^2=\frac{2(\eta_0-\bar{\nu}^2\bar{\eta}_0)}{(1+\nu\bar{\nu})^2}+\frac{2\nu}{1+\nu\bar{\nu}}r_0,
\end{equation}
\begin{equation}\label{e:2b}
x^3=-\frac{2(\nu\bar{\eta}_0+\bar{\nu}\eta_0)}{(1+\nu\bar{\nu})^2}+\frac{1-\nu\bar{\nu}}{1+\nu\bar{\nu}}r_0.
\end{equation}
where $r_0:S\rightarrow{\mathbb R}$ is the support function of $S$. 

To find the tangent line to $c$, differentiate equations (\ref{e:2a}) and (\ref{e:2b}) with respect to $u$ to find
\[
\dot{z}=\frac{2}{(1+\nu\bar{\nu})^2}\left[(\psi_0+\sigma_0\nu^2)\dot{\nu}-(\psi_0\nu^2+\sigma_0)\dot{\bar{\nu}}\right]
\]
\[
\dot{x^3}=-\frac{2}{(1+\nu\bar{\nu})^2}\left[(\psi_0\bar{\nu}-\sigma_0\nu)\dot{\nu}+(\psi_1\xi_1-\bar{\sigma}_0\bar{\nu})\dot{\bar{\nu}}\right],
\]
where have substituted for the derivatives of $\eta_0$ and $r_0$ using equation (\ref{e:pot}) and the definitions of $\sigma_0$ and $\psi_0$ which yield: 
\[
\dot{\eta_0}=\frac{\partial\eta_0}{\partial\nu}\dot{\nu}+\frac{\partial\eta_0}{\partial\bar{\nu}}\dot{\bar{\nu}}
=\left(\psi_0-r_0+\frac{2\bar{\nu}\eta_0}{1+\nu\bar{\nu}}\right)\dot{\nu}-\bar{\sigma}_0\dot{\bar{\nu}}
\]
\[
\dot{r_0}=\frac{\partial r_0}{\partial\nu}\dot{\nu}+\frac{\partial r_0}{\partial\bar{\nu}}\dot{\bar{\nu}}
=\frac{2\bar{\eta}_0}{(1+\nu\bar{\nu})^2}\dot{\nu}+\frac{2\eta_0}{(1+\nu\bar{\nu})^2}\dot{\bar{\nu}}.
\]
The curve is parameterized by arclength iff
\[
|\vec{T}|^2=\dot{z}\dot{\bar{z}}+(\dot{x^3})^2=\frac{4}{(1+\nu\bar{\nu})^2}
    \left|\psi_0\dot{\nu}-\bar{\sigma}_0\dot{\bar{\nu}}\right|^2=1,
\]
where $\vec{T}$ is the tangent vector to $c$.

That is, there exists $\hat{\beta}\in [0,2\pi)$ such that 
\[
\psi_0\dot{\nu}-\bar{\sigma}_0\dot{\bar{\nu}}=\frac{1}{2}(1+\nu\bar{\nu})e^{i\hat{\beta}},
\]
inverting this last equation (with the aid of its conjugate)
\[
\dot{\nu}=\frac{(1+\nu\bar{\nu})}{2(\psi_0^2-|\sigma_0|)}\left[\psi_0e^{i\hat{\beta}}+\bar{\sigma}_0e^{-i\hat{\beta}}\right].
\]
Now comparing this with
\[
\vec{T}=\dot{z}\frac{\partial }{dz}+\dot{\bar{z}}\frac{\partial }{d\bar{z}}+\dot{x^3}\frac{\partial }{dx^3}
=\frac{2\xi}{1+\xi\bar{\xi}}\frac{\partial }{dz}+\frac{2\bar{\xi}}{1+\xi\bar{\xi}}\frac{\partial }{d\bar{z}}
   +\frac{1-\xi\bar{\xi}}{1+\xi\bar{\xi}}\frac{\partial }{dx^3}
\]
where $\xi$ is given by equation (\ref{e:2}), we find that $\hat{\beta}=A$. Thus, the tangent to the curve ${\mathcal C}\subset{\mathcal H}(S)$ at a point
is of the form  
\[
\vec{X}=a\frac{\partial}{\partial A}+b{\mathbb R}{\mbox{e}}\left[e^{iB}\frac{\partial}{\partial \nu}\right],
\]
for $a,b\in{\mathbb R}$ with
\[
B=A+\frac{1}{2i}\ln\left(\frac{\psi_0+\bar{\sigma}_0e^{-2iA}}{\psi_0+{\sigma}_0e^{2iA}}\right).
\]
By Lemma \ref{l:albe}, this vector is contained in an $\alpha$-plane, as claimed.

On the other hand, the normal $\vec{N}$ to the curve $c$ gives rise to the vector $\vec{X}$ with
\[
B=A+\frac{1}{2i}\ln\left(\frac{\psi_0+\bar{\sigma}_0e^{-2iA}}{\psi_0+{\sigma}_0e^{2iA}}\right)+\frac{\pi}{2},
\]
and this is contained in a $\beta$-plane iff either $\sigma_0=0$, in which case the point is umbilic, or if ${\sigma}_0e^{2iA}$ is real, in
which case the curve $c$ is a line of curvature. 

Similarly, if ${\mathcal C}$ is $\beta$-Legendrian, then the oriented lines are normal
to $c$ iff either $\sigma_0=0$, in which case the point is umbilic, or if ${\sigma}_0e^{2iA}$ is real, in
which case the curve $c$ is a line of curvature. 

To prove the final part of Theorem \ref{t2} consider the contact 1-form $\omega^+$ defined in equation (\ref{e:alcontf}). 

The Reeb vector field associated with $\omega^+$ is easily found to be
\begin{align}
X=&\frac{i(1+\nu\bar{\nu})}{2(\psi_0^2-\sigma_0\bar{\sigma}_0)}\left[(\psi_0e^{iA}-\bar{\sigma}_0e^{-iA})\frac{\partial}{\partial\nu}+
  (\psi_0e^{-iA}-{\sigma}_0e^{iA})\frac{\partial}{\partial\bar{\nu}}\right]\nonumber\\
&\qquad+\frac{1}{2(\psi_0^2-\sigma_0\bar{\sigma}_0)}
\left[(\psi_0\bar{\nu}-\sigma_0\nu)e^{iA}+(\psi_0\nu-\bar{\sigma}_0\bar{\nu})e^{-iA})\right]\frac{\partial}{\partial A}\nonumber
\end{align}

We conclude that flowing by the Reeb vector using a parameter $r$ leads to the flow
\[
\frac{d \nu}{dr}=\frac{i(1+\nu\bar{\nu})}{2(\psi_0^2-|\sigma_0|^2)}\left(\psi_0e^{iA}-\bar{\sigma}_0e^{-iA}\right)
\]
\[
\frac{d A}{dr}=\frac{1}{2(\psi_0^2-|\sigma_0|^2)}
\left[\left(\psi_0\bar{\nu}-\sigma_0\nu\right)e^{iA}+\left(\psi_0\nu-\bar{\sigma}_0\bar{\nu}\right)e^{-iA}\right]
\]

This flow can be understood by considering the geodesic flow on $S$ which induces the following flow on ${\mathcal H}(S)$:
\[
\frac{d \nu}{d\tau}=\frac{(1+\nu\bar{\nu})}{2(\psi_0^2-|\sigma_0|^2)}\left(\psi_0e^{iA}+\bar{\sigma}_0e^{-iA}\right)
\]
\[
\frac{d A}{d\tau}=-\frac{i}{2(\psi_0^2-|\sigma_0|^2)}
\left[\left(\psi_0\bar{\nu}-\sigma_0\nu\right)e^{iA}-\left(\psi_0\nu-\bar{\sigma}_0\bar{\nu}\right)e^{-iA}\right]
\]

The Reeb flow is obtained from the geodesic flow by replacing $A$ by $A+\pi/2$. Thus integral curves of the Reeb flow consists of the  
oriented lines along a geodesic of $S$ that are orthogonal to the geodesic.

This completes the proof of Theorem \ref{t2} in the flat case.

\vspace{0.1in}


\subsection{The non-flat case}

\subsubsection{The neutral metric}

For $\epsilon\in\{-1,1\}$ consider the following flat metrics in ${\mathbb R}^4$:
\[
\left<.,.\right>_{\epsilon}=\epsilon(dx^1)^2+\epsilon(dx^2)^2+\epsilon(dx^3)^2+(dx^4)^2.
\]
Let ${\mathbb S}_{\epsilon}^3=\{x\in{\mathbb R}^4\; :\; \left<x,x\right>_{\epsilon}=1\}$ be the 3-(pseudo)-sphere in the Euclidean space 
${\mathbb R}_{\epsilon}^4:=({\mathbb R}^4,\left<.,.\right>_{\epsilon})$. Note that ${\mathbb S}_{1}^3$ is the standard 3-sphere, while 
${\mathbb S}_{-1}^3$ is anti-isometric to the hyperbolic 3-space ${\mathbb H}^3$. 

Let $\iota:{\mathbb S}_{\epsilon}^3\hookrightarrow {\mathbb R}^4$ be the inclusion map and denote by $g_{\epsilon}$ the induced metric 
$\iota^{\ast}\left<.,.\right>_{\epsilon}$. The space of oriented geodesics ${\mathbb L}({\mathbb S}_{\epsilon}^3)$ of 
$({\mathbb S}_{\epsilon}^3,g_{\epsilon})$ is 4-dimensional and ${\mathbb L}({\mathbb S}_{1}^3)$ can be identified with the Grasmannian 
of oriented planes in ${\mathbb R}_1^4$, while ${\mathbb L}({\mathbb S}_{-1}^3)$ can be identified with the Grasmannian of oriented planes 
in ${\mathbb R}_{-1}^4$ such that the induced metric is Lorentzian \cite{An}. 

Thus, ${\mathbb L}({\mathbb S}_{\epsilon}^3)$ is the following submanifold of the space $\Lambda^2({\mathbb R}^4)$ of bivectors in ${\mathbb R}^4$:
$${\mathbb L}({\mathbb S}_{\epsilon}^3)=\{x\wedge y\in\Lambda^2({\mathbb R}^4)\; :\; y\in \mbox{T}_x{\mathbb S}_{\epsilon}^3,\; 
\left<y,y\right>_{\epsilon}=\epsilon\}.$$
In fact, an element $x\wedge y\in {\mathbb L}({\mathbb S}_{\epsilon}^3)$ is the oriented geodesic $\gamma\subset {\mathbb S}_{\epsilon}^3$ 
passing through $x\in {\mathbb S}_{\epsilon}^3$ and has direction $y\in \mbox{T}_x{\mathbb S}_{\epsilon}^3$ with $\left<y,y\right>_{\epsilon}=\epsilon$.

Endow $\Lambda^2({\mathbb R}^4)$ with the flat metric $\left<\left<.,.\right>\right>_{\epsilon}$ defined by:
\[
\left<\left<x_1\wedge y_1,x_2\wedge y_2\right>\right>_{\epsilon}=\left<x_1,x_2\right>_{\epsilon}
\left<y_1,y_2\right>_{\epsilon}-\left<x_1,y_2\right>_{\epsilon}\left<y_1,x_2\right>_{\epsilon}.
\] 
If $x\wedge y\in {\mathbb L}({\mathbb S}_{\epsilon}^3)$, the tangent space $\mbox{T}_{x\wedge y}{\mathbb L}({\mathbb S}_{\epsilon}^3)$ 
is the vector space  consisting of vectors of the form $x\wedge X+y\wedge Y$, where 
$X,Y\in (x\wedge y)_{\epsilon}^{\bot}=\{\xi\in \Lambda^2({\mathbb R}^4)\; :\; \left<\xi,x\right>_{\epsilon}=\left<\xi,y\right>_{\epsilon}=0\}$.

A complex (resp. paracomplex) structure $J$ can be defined in the oriented plane $x\wedge y\in {\mathbb L}({\mathbb S}_{1}^3)$ 
(resp. ${\mathbb L}({\mathbb S}_{-1}^3)$) by $Jx=y$ and $Jy=-x$ (resp. $Jy=x$) and let $J'$ be the complex structure on 
the oriented plane $(x\wedge y)_{\epsilon}^{\bot}$. Define the endomorphisms ${\mathbb J}$ and ${\mathbb J}'$ on 
$T_{x\wedge y}{\mathbb L}({\mathbb S}_{\epsilon}^3)$ as follows:
\[
{\mathbb J}(x\wedge X+y\wedge Y)=Jx\wedge X+Jy\wedge Y=y\wedge X-\epsilon x\wedge Y,
\]
and 
\[
{\mathbb J}'(x\wedge X+y\wedge Y)=x\wedge J'(X)+y\wedge J'(Y).
\]
For $\epsilon=1$ (resp. $\epsilon=-1$) ${\mathbb J}$ is a complex (resp. paracomplex) structure on ${\mathbb L}({\mathbb S}_{\epsilon}^3)$, while 
${\mathbb J}'$ is a complex structure for $\epsilon=\pm1$ \cite{agk} \cite{An} \cite{CU}.

Denoting the inclusion map by $\iota:{\mathbb L}({\mathbb S}_{\epsilon}^3)\hookrightarrow \Lambda^2({\mathbb R}^4)$, the metric 
$\iota^{\star}\left<\left<.,.\right>\right>_{\epsilon}$ is Riemannian and Einstein \cite{Leich}. The metric 
${\mathbb G}_{\epsilon}=-\iota^{\star}\left<\left<.,{\mathbb J}\circ {\mathbb J}'.\right>\right>_{\epsilon}$ is of neutral signature, 
locally conformally flat and is invariant under the natural action of the group $SO\displaystyle((7+\epsilon)/2,(1-\epsilon)/2)$ of 
isometries of ${\mathbb S}^3_{\epsilon}$. Additionally, both structures 
$({\mathbb L}({\mathbb S}_{\epsilon}^3),{\mathbb J},\iota^{\star}\left<\left<.,.\right>\right>_{\epsilon})$ and 
$({\mathbb L}({\mathbb S}_{\epsilon}^3),{\mathbb J}',{\mathbb G}_{\epsilon})$ are (para-) K\"ahler manifolds \cite{agk} \cite{An} 
\cite{gag1} \cite{honda}.

\vspace{0.2in}
\subsubsection{Tangent hypersurfaces}

Consider an oriented smooth surface $S$ of ${\mathbb S}_{\epsilon}^3$ given by the immersion $\phi:\overline{S}\rightarrow {\mathbb S}_{\epsilon}^3$, with $S=\phi(\overline{S})$. Let $(e_1,e_2)$ be an oriented orthonormal frame of the tangent bundle of $S$ and let $N$ be the  unit normal vector field such that $(\phi,e_1,e_2,N)$ is a positive oriented orthonormal frame in ${\mathbb R}_{\epsilon}^4$. Then 
\[
\left<\phi,\phi\right>_{\epsilon}=\epsilon\left<e_1,e_1\right>_{\epsilon}=\epsilon\left<e_2,e_2\right>_{\epsilon}=\epsilon\left<N,N\right>_{\epsilon}=1.
\]

 For $\theta\in {\mathbb S}^1$, define the following tangential vector fields
\[
v(x,\theta)=\cos\theta\; e_1+\sin\theta\; e_2,\qquad v^{\bot}(x,\theta)=-\sin\theta\; e_1+\cos\theta\; e_2.
\]
As in the flat case, the tangent hypersurface $\mathcal{H}(S)$ in ${\mathbb L}({\mathbb S}_{\epsilon}^3)$ is the image of the immersion 
$\overline{\phi}:\overline{S}\times {\mathbb S}^1\rightarrow {\mathbb L}({\mathbb S}_{\epsilon}^3):(x,\theta)\mapsto\phi(x)\wedge v(x,\theta).$

Identify $e_i$ with $d\phi(e_i)$ and assume that $(e_1,e_2)$ diagonalize the shape operator, that is, $h(e_i,e_j)=k_i\delta_{ij}$, where $k_i$  and 
$h$ denote the principal curvatures and second fundamental form, respectively. 

If $\nabla$ denotes the Levi-Civita connection of the induced metric $\phi^{\ast} g_{\epsilon}$ and setting $v_1:=\left<\nabla_{e_1}v,v^{\bot}\right>_{\epsilon}$ and $v_2:=\left<\nabla_{e_2}v,v^{\bot}\right>_{\epsilon}$, the derivative of $\overline\phi$ is given by:
\begin{equation}\label{e:equa1}
 \left.\begin{aligned}
      d\overline\phi(e_1)&=v_1\phi\wedge v^{\bot}+k_1\cos\theta\;\phi\wedge N+\sin\theta\; v\wedge v^{\bot}\\
d\overline\phi(e_2)&=v_2\phi\wedge v^{\bot}+k_2\sin\theta\;\phi\wedge N-\cos\theta\; v\wedge v^{\bot}\\
 d\overline\phi(\partial/\partial\theta)&=\phi\wedge v^{\bot}.\end{aligned}
 \right\} \\
\end{equation}
A direct computation shows that 
$${\mathbb G}_{\epsilon}(d\overline\phi(\partial/\partial\theta), d\overline\phi(e_1))={\mathbb G}_{\epsilon}(d\overline\phi(\partial/\partial\theta), 
d\overline\phi(e_2))=0.$$ 
 In addition, $d\overline\phi(\partial/\partial\theta)$ is null, that is, 
\begin{eqnarray}
{\mathbb G}_{\epsilon}(d\overline\phi(\partial/\partial\theta),d\overline\phi(\partial/\partial\theta))&=& 0.\nonumber
\end{eqnarray}
Now, a brief computation gives
\[
{\mathbb G}_{\epsilon}( d\overline\phi\, e_1, d\overline\phi\, e_1){\mathbb G}_{\epsilon}( d\overline\phi\,e_2, d\overline\phi\,e_2)
-{\mathbb G}_{\epsilon}(d\overline\phi\,e_1, d\overline\phi\,e_2)^2=-(k_2\sin^2\theta+k_1\cos^2\theta)^2.
\]
Thus, $d\overline\phi(\partial/\partial\theta)$ is a tangential vector field and a normal vector field of the hypersurface 
$\mathcal{H}(S)$. The induced metric $\overline\phi^{\ast}{\mathbb G}_{\epsilon}$ is of signature $(+-0)$.

Let $\rho_1=d\overline\phi(e_1)$ and $\rho_2$ be defined by
\begin{equation}\label{e:rho2}
\rho_2=\frac{2k_1v_2\cos\theta\sin\theta+(k_1\cos^2\theta-k_2\sin^2\theta)v_1}{k_1\cos^2\theta+k_2\sin^2\theta}\phi\wedge v^{\bot}
 +k_1\cos\theta\;\phi\wedge N-\sin\theta\; v\wedge v^{\bot}.
\end{equation}
Consider the null vectors $e_+$ and $e_-$ by $e_+=\rho_1+\rho_2$ and $e_-=\rho_1-\rho_2$. 

If $e_0=d\overline\phi(\partial/\partial\theta)$, define the null planes $\Pi_+:=\mbox{span}\{e_+,e_0\}$ and 
$\Pi_-:=\mbox{span}\{e_-,e_0\}$. A brief computation shows that 
\[
\Pi_+=\mbox{span}\{\phi\wedge v^{\bot},\phi\wedge N\}\;\;\mbox{and}\;\; \Pi_-=\mbox{span}\{\phi\wedge v^{\bot}, v\wedge v^{\bot}\}.
\]

\begin{Prop}\label{p:alphaplane}
The plane $\Pi_{+}$ is an $\alpha$-plane, while $\Pi_{-}$ is a $\beta$-plane.
\end{Prop}
\begin{proof}
If $\xi\in \Pi_+$, we have that $\xi=\xi_1\; \phi\wedge v^{\bot}+\xi_2\; \phi\wedge N$ and thus,
\[
{\mathbb J}'\xi=-\xi_1\;\phi\wedge N+\xi_2\; \phi\wedge v^{\bot}\in\Pi_+.
\]
Therefore, the null plane $\Pi_+$ is ${\mathbb J}'$-holomorphic and since it is totally null, it is therefore an $\alpha$-plane.

If $\xi\in \Pi_-$ we have that $\xi=\xi_1\; \phi\wedge v^{\bot}+\xi_2\; v\wedge v^{\bot}$. Then,
\[
{\mathbb J}\xi=\xi_1\; v\wedge v^{\bot}-\epsilon\xi_2\; \phi\wedge v^{\bot}\in\Pi_-,
\]
which shows that $\Pi_-$ is ${\mathbb J}$-holomorphic, and thus, $\Pi_-$ is a $\beta$-plane. 
\end{proof}

\vspace{0.1in}

The following Proposition establishes the first part of Theorem \ref{t2} in the non-flat cases:

\vspace{0.1in}
 
\begin{Prop}
Let $S$ be a smooth oriented convex surface in ${\mathbb S}_{\epsilon}^3$ and let $\mathcal{H}(S)$ be its tangent hypersurface. Then, $(\mathcal{H}(S),\Pi_+)$ and $(\mathcal{H}(S),\Pi_-)$ are both contact 3-manifolds.
\end{Prop}
\begin{proof}
Assuming that $S$ is convex, we have that $k_1k_2>0$ and thus 
\[
k_1(x)\cos^2\theta+k_2(x)\sin^2\theta\neq 0,\qquad\forall\; (x,\theta)\in \mathcal{H}(S).
\]
Set $\eta_1=\phi\wedge v^{\bot}, \eta_2=\phi\wedge N$ and $\eta_3=v\wedge v^{\bot}$. We simply write $e_i$ for the tangential vector fields 
$d\bar\phi (e_i)$ and $\partial/\partial\theta$ for the tangential vector field $d\bar\phi (\partial/\partial\theta)$. Then solving the relations 
(\ref{e:equa1}) for $\eta_i$ we have 
\[
\eta_1=\partial/\partial\theta
\]
\[
\eta_2=\frac{cos\theta}{k_1\cos^2\theta+k_2\sin^2\theta}e_1+\frac{\sin\theta}{k_1\cos^2\theta+k_2\sin^2\theta}e_2-
\frac{v_1\cos\theta+v_2\sin\theta}{k_1\cos^2\theta+k_2\sin^2\theta}\partial/\partial\theta
\]
\[
\eta_3=\frac{k_2\sin\theta}{k_1\cos^2\theta+k_2\sin^2\theta}e_1-\frac{k_1\cos\theta}{k_1\cos^2\theta+k_2\sin^2\theta}e_2-
\frac{v_1k_2\sin\theta-v_2k_1\cos\theta}{k_1\cos^2\theta+k_2\sin^2\theta}\partial/\partial\theta.
\]

Thus, $\{\eta_1,\eta_2,\eta_3\}$ are tangential vector fields and let $\eta^{i}$ be the dual orthonormal frame. Then 
$\eta^i\eta_j=\delta^i_j$ and $\eta^{i}\in T^{\ast}(\mathcal{H}(S))$.

Observe that $\Pi_+$ is generated by the vectors $\eta_1,\eta_2$, and thus $\eta^3(\Pi_+)=0$.

If $(e^1,e^2,d\theta)$ is the dual frame of $(e_1,e_2,\partial/\partial\theta)$ we have
\[
\eta^3=\sin\theta\, e^1-\cos\theta\, e^2.
\]
Hence,
\begin{equation}\label{e:eta3anaposo}
\eta^3\wedge d\eta^3=e^1\wedge e^2\wedge d\theta.
\end{equation}
which implies that $\eta^3\wedge d\eta^3\neq 0$ and thus $(\mathcal{H}(S),\Pi_+)$ is a contact manifold.

The $\beta$-plane $\Pi_-$ is generated by the vectors $\eta_1,\eta_3$, and thus 
$\eta^2(\Pi_-)=0$. A brief computation gives
\begin{equation}\label{e:eta2anaposo}
\eta^2=k_1\cos\theta\, e^1+k_2\sin\theta\, e^2,
\end{equation}
and then,
\[
\eta^2\wedge d\eta^2=-k_1k_2\, e^1\wedge e^2\wedge d\theta.
\]
Using the fact that $S$ is convex, it follows that $\eta^2\wedge d\eta^2\neq 0$ and thus $(\mathcal{H}(S),\Pi_-)$ is a contact manifold.
\end{proof}

\vspace{0.1in}

For any smoothly embedded convex surface of $S\subset {\mathbb S}^3_{\epsilon}$ consider the constant angle hypersurface 
${\mathcal H}_a(S)$ of ${\mathbb L}({\mathbb S}^3_{\epsilon})$ which is the set of all oriented geodesics passing through $S$ and  
making an angle $a$ with the normal vector field $N$ of $S$. As in the flat case, ${\mathcal H}_0(S)$ is a Lagrangian surface in 
${\mathbb L}({\mathbb S}^3_{\epsilon})$, while ${\mathcal H}_{\pi/2}(S)$ is the tangent hypersurface. The following Proposition covers the other cases:
\begin{Prop}
For $a\in (0,\pi/2)$, the hypersurface ${\mathcal H}_a(S)$ is null exactly 
\begin{enumerate}
\item at the oriented geodesics passing through an umbilic point on $S$ and,
\item at the oriented geodesics whose directions projected to the tangent bundle $TS$ are tangent to the lines of curvature of $S$.
\end{enumerate} 
\end{Prop}
\begin{proof}
Let $\phi$ be an immersion of $S$ in ${\mathbb S}^3_{\epsilon}$ and consider, as before, the oriented orthonormal frame $(\phi,e_1,e_2,N)$, where 
$(e_1,e_2)$ are the principal directions.

For $a\in (0,\pi/2)$, the hypersurface ${\mathcal H}_a(S)$ is given by the image of the immersion 
\[
\bar\phi_a(x,\theta)=\phi(x)\wedge v_a(x,\theta),
\] 
where, $x\in S$ and $\theta\in [0,2\pi)$ with
\[
v_a(x,\theta)=(\cos\theta\, e_1(x)+\sin\theta\,e_2(x))\sin a+N\cos a.
\]
Consider the following normal vector field of ${\mathcal H}_a(S)$:
\[
\overline{N}_a=-\Big(\cos\theta\left<\overline\nabla_{e_1}v_a,\xi_2\right>_{\epsilon}
+\sin\theta\left<\overline\nabla_{e_2}v_a,\xi_2\right>_{\epsilon}\Big)\,\phi\wedge\xi_1+\cos a\, v_a\wedge\xi_1
\]
\[
\qquad\qquad\qquad\qquad\qquad\qquad\qquad-\cos a\Big(\sin\theta\left<\overline\nabla_{e_1}v_a,\xi_2\right>_{\epsilon}
-\cos\theta\left<\overline\nabla_{v_a}v_a,\xi_2\right>_{\epsilon}\Big)\,\phi\wedge\xi_2,
\]
where $\overline\nabla$ denotes the Levi-Civita connection of $g_{\epsilon}$. Then,
\[
{\mathbb G}(\overline{N}_a,\overline{N}_a)=(k_2-k_1)\cos^2a\sin2\theta,
\]
and the Proposition follows.
\end{proof}

\vspace{0.1in}

Let $S$ be a smooth convex 2-sphere in ${\mathbb S}_{\epsilon}^3$ and let ${\mathcal C}:I\rightarrow \mathcal{H}(S):u\mapsto \phi(u)\wedge v(u)$ 
be an $\alpha$-Legendrian curve, where the curve $\phi(u)$ in $S$ is parametrised by the arclength $u$. 

By definition we have $\dot{{\mathcal C}}=\dot\phi\wedge v+\phi\wedge \dot{v}\in \mbox{span}\{\phi\wedge v^{\bot}, \phi\wedge N\}$, and thus there 
exist two real functions $\lambda_1$ and $\lambda_2$ such that 
\[
\dot\phi\wedge v+\phi\wedge \dot{v}=\lambda_1\phi\wedge v^{\bot}+\lambda_2\phi\wedge N.
\]
We have
\begin{equation}\label{e:equat11}
\dot\phi\wedge v+\phi\wedge(\dot{v}-\lambda_1 v^{\bot}-\lambda_2 N)=0.
\end{equation}
If $\dot\phi=a_1v+a_2\phi+a_3 v^{\bot}+a_4 N$, it is obvious that $a_3=a_4=0$. Then $\dot\phi=a_1v+a_2\phi$ and since 
$\left<\phi,\dot\phi\right>_{\epsilon}=0$ we have that $a_2=0$. Then $\dot\phi=a_1v$ and since $|\dot\phi|_{\epsilon}^2=\epsilon$ we have that 
either $a_1=1$ or $a_1=-1$. In any case, $v=a_1\dot\phi$ and thus,
\[
{\mathcal C}=\phi\wedge v=\phi\wedge a_1\dot\phi,
\]
where $a_1^2=1$.

We turn now to the proof of the second part of Theorem \ref{t2} in the non-flat case in 3 steps.

\newpage

\noindent{\bf (i) and (ii) imply (iii)}:

The fact that ${\mathcal C}$ is $\beta$-Legendrian implies that there exist functions $a,b$ along the curve $\phi$ such that,
\begin{equation}\label{e:eqcdot}
\dot{{\mathcal C}}=\dot\phi\wedge v+\phi\wedge\dot{v}=a\phi\wedge v^{\perp}+bv\wedge v^{\perp},
\end{equation}
and since ${\mathcal C}=\phi\wedge v$ is normal to the curve $\phi$ we have that 
\[
\left<\dot\phi, v\right>_{\epsilon}=0.
\]
Since $N$ is the unit normal vector field of $S$ $\left<\dot\phi, N\right>_{\epsilon}=0$, and therefore
$\dot\phi=\pm v^{\perp}$.
Now, (\ref{e:eqcdot}) yields $\phi\wedge (\dot{v}-av^{\perp})=0$, which implies $\dot{v}=\mu\phi+av^{\perp}$.
Then 
\begin{equation}\label{e:eqcdotN}
\left<\dot{v}, N\right>_{\epsilon}=0,
\end{equation}
and since $$0=\left<\dot{N},\phi\right>_{\epsilon}=-\left<\dot{\phi},N\right>_{\epsilon}=\left<v^{\perp},N\right>_{\epsilon},$$
we have that $\dot{N}=\lambda v^{\perp}=\lambda\dot\phi$ and therefore $\phi$ is a line of curvature.

On the other hand,
\begin{align}
0=\left<\dot{v},N\right>_{\epsilon}&= -\left<v,\nabla_{v^{\bot}}N\right>_{\epsilon}
= \left<\cos\theta e_1+\sin\theta e_2, A(-\sin\theta e_1+\cos\theta e_2)\right>_{\epsilon}\nonumber\\
&=\left<\cos\theta e_1+\sin\theta e_2, -\sin\theta Ae_1+\cos\theta Ae_2\right>_{\epsilon}= \epsilon\,(k_2-k_1)\cos\theta\sin\theta,\nonumber
\end{align}
which shows that $S$ is umbilic along the curve $\phi$.

\vspace{0.1in}

\noindent{\bf (i) and (iii) imply (ii)}:

The fact that ${\mathcal C}$ is $\beta$-Legendrian gives (\ref{e:eqcdot}). Suppose that the curve $\phi$ is also a line of curvature. Then 
$\dot{N}=\lambda\dot\phi$, where $\lambda$ is a non zero function along the curve. We also have,
\[
\dot\phi=a_1v+a_2\bot v^{\bot}.
\]
From (\ref{e:eqcdot}) we have
\[
a_2v^{\perp}\wedge v+\phi\wedge\dot{v}=a\phi\wedge v^{\perp}+b v\wedge v^{\perp},
\]
which gives $\phi\wedge (\dot{v}-av^{\perp})=0$, and thus $\dot{v}=av^{\perp}+\mu\phi$.
Then, $\left<\dot{v}, N\right>_{\epsilon}=0$, which yields,
\[
0=\left<\dot{v}, N\right>_{\epsilon}=-\left<v,\dot{N}\right>_{\epsilon}=-\lambda\left<v,\dot{\phi}\right>_{\epsilon}.
\]
It follows that $\left<\dot\phi, v\right>_{\epsilon}=0$ and hence ${\mathcal C}=\phi\wedge v$ is normal to the curve $\phi$.

Suppose now that $S$ is umbilic along the curve $\phi=\phi(u)$, i.e., $k_1=k_2$. The relation (\ref{e:eqcdot}) implies
\[
(\dot\phi+bv^{\perp})\wedge v+\phi\wedge (\dot{v}-av^{\perp})=0,
\]
which gives the following equations:
\[
\dot\phi=-bv^{\perp}+\mu v \qquad{\mbox and}\qquad \dot{v}=av^{\perp}+s\phi.
\]
Then $\left<\dot{v}, N\right>_{\epsilon}=0$ and hence,
\begin{align}
0&=\left<v,\dot{N}\right>_{\epsilon}=\left<v,\nabla_{\dot{\phi}}N\right>_{\epsilon}=\left<v,\nabla_{-bv^{\perp}+\mu v}N\right>_{\epsilon}\nonumber\\
&=-b\left<v,\nabla_{v^{\perp}}N\right>_{\epsilon}+\mu\left<v,\nabla_{v}N\right>_{\epsilon}.\nonumber
\end{align}
The fact that $S$ is umbilic along the curve, implies that $\left<v,\nabla_{v^{\perp}}N\right>_{\epsilon}=0$. Then
\[
0=\mu\left<v,\nabla_{v}N\right>_{\epsilon}=-\epsilon\mu(k_1\cos^2\theta+k_2\sin^2\theta),
\]
and since $S$ is convex, we have that $k_1\cos^2\theta+k_2\sin^2\theta\neq 0$. It follows that $\mu=0$ and thus $\dot\phi=-b v^{\perp}$.
Therefore, $\left<\dot\phi, v\right>_{\epsilon}=0$ and hence ${\mathcal C}=\phi\wedge v$ is normal to the curve $\phi=\phi(u)$.

\vspace{0.1in}

\noindent{\bf (ii) and (iii) imply (i)}:

The fact that ${\mathcal C}$ is normal to $\phi=\phi(u)$ implies that $\left<\dot\phi,v\right>_{\epsilon}=0$. Suppose that the curve 
$\phi$ is a line of curvature. Then $\dot{N}=\lambda\dot\phi$, where $\lambda$ is a non zero function along the curve. We also have that,
\begin{equation}\label{e:phidot1}
\dot\phi=a_1v+b_1 v^{\perp}.
\end{equation}
Since $$\left<\dot{v},N\right>_{\epsilon}=-\left<\dot{N},v\right>_{\epsilon}=-\lambda\left<\dot{\phi},v\right>_{\epsilon}=0,$$
we obtain
\begin{equation}\label{e:phiv1}
\dot{v}=a_2\phi+b_2 v^{\perp}.
\end{equation}
Using (\ref{e:phidot1}) and (\ref{e:phiv1}) we have:
\begin{align}
\dot{{\mathcal C}}&=\dot\phi\wedge v+\phi\wedge\dot{v}=(a_1v+b_1 v^{\perp})\wedge v+\phi\wedge(a_2\phi+b_2 v^{\perp})\nonumber\\
&=-b_1 v\wedge v^{\perp}+b_2 \phi\wedge v^{\perp}\in \Pi_-,\nonumber
\end{align}
and thus ${\mathcal C}$ is $\beta$-Legendrian.

\vspace{0.1in}

 Suppose that $S$ is umbilic along the curve $\phi$ and that ${\mathcal C}$ is normal to $\phi=\phi(u)$. Then $\left<\dot\phi,v\right>_{\epsilon}=0$ and 
hence the equation (\ref{e:phidot1}) becomes $\dot\phi= b_1v^{\perp}$.
It follows that $\left<\dot{v},\phi\right>_{\epsilon}=-\left<\dot{\phi},v\right>_{\epsilon}=0$ and
\begin{align}
\left<\dot{v},N\right>_{\epsilon}&=-\left<v,\dot{N}\right>_{\epsilon}=-\left<v,\nabla_{\dot\phi}N\right>_{\epsilon}\nonumber\\
&=-\left<v,\nabla_{v^{\perp}}N\right>_{\epsilon}=\epsilon\, (k_1-k_2)\cos\theta\sin\theta=0.\nonumber
\end{align}
Thus,
\[
\dot\phi=b_1v^{\perp} \qquad{\mbox and}\qquad \dot{v}=b_2v^{\perp}.
\]
Therefore, 
\[
\dot{{\mathcal C}}=\dot\phi\wedge v+\phi\wedge\dot{v}=-b_1 v\wedge v^{\perp}+b_2\phi\wedge  v^{\perp}\in \Pi_-,
\]
which shows again that ${\mathcal C}$ is $\beta$-Legendrian.

We prove the final part of Theorem \ref{t2} for the case of ${\mathbb L}({\mathbb S}^3)$ while, the proof for the case of ${\mathbb L}({\mathbb H}^3)$ is similar. Consider the contact 1-form $\eta^3$ of the contact manifold $({\mathcal H}(S),\Pi_+)$ given in equation (\ref{e:eta3anaposo}). 

A brief computation shows that the Reeb vector field $X$ associated with $\eta^3$ is 
\[
X=\left<v,\nabla_{v^{\bot}}v^{\bot}\right>\phi\wedge v^{\bot}+(k_1-k_2)\cos\theta\sin\theta\,\phi\wedge N+v\wedge v^{\bot}.
\]

Let $C(t)=\phi(t)\wedge v(t)$ be a smooth regular curve in $\mathcal{H}(S)$, where $t$ is the arclength of the contact curve $\phi=\phi(t)$ and for every $t$, the velocity $\dot{C}(t)$ is a Reeb vector. It then follows,
\[
\dot\phi\wedge v+\phi\wedge\dot{v}=\left<v,\nabla_{v^{\bot}}v^{\bot}\right>\phi\wedge v^{\bot}+(k_1-k_2)\cos\theta\sin\theta\,\phi\wedge N+v\wedge v^{\bot},
\]
which yields,
\begin{equation}\label{e:flowequat}
\dot\phi=-v^{\bot}\qquad\qquad \dot{v}=\left<v,\nabla_{v^{\bot}}v^{\bot}\right>v^{\bot}+(k_1-k_2)\cos\theta\sin\theta\, N,
\end{equation}
Thus, 
\[
\left<\dot\phi,v\right>=-\left<v^{\bot},v\right>=0.
\]
Therefore, the curve $C(t)$ is formed by the oriented geodesics that are orthogonal to the contact curve $\phi$ and therefore we have proved the first statement.

Using that $\left<\dot\phi,\dot\phi\right>=1$, we have
\begin{equation}\label{e:geodesicequat}
\ddot{\phi}=-\dot{v}^{\bot}=-\phi+\left<v,\nabla_{v^{\bot}}v^{\bot}\right>v+(k_1\sin^2\theta+k_2\cos^2\theta)\,N.
\end{equation}
Denoting the vector fields $d\phi(\partial/\partial t),\,d\phi(\partial/\partial\theta)$ by $\partial/\partial t,\,\partial/\partial\theta$, respectively and using (\ref{e:flowequat}), we have
\begin{equation}\label{e:dfff}
\nabla_{\partial/\partial\theta}\nabla_{v^{\bot}}=-\nabla_{\partial/\partial\theta}\nabla_{\partial/\partial t}=-\nabla_{\partial/\partial t}\nabla_{\partial/\partial\theta}=\nabla_{v^{\bot}}\nabla_{\partial/\partial\theta}.
\end{equation}
Note that
\[
\left<v,\nabla_{e_1}v^{\bot}\right>=\left<e_1,\nabla_{e_1}e_2\right>\qquad 
\left<v,\nabla_{e_2}v^{\bot}\right>=\left<e_2,\nabla_{e_2}e_1\right>,
\]
and thus for $i=1,2$,
\[
(\partial/\partial\theta)\left<v,\nabla_{e_1}v^{\bot}\right>=0,
\qquad
(\partial/\partial\theta)\left<v,\nabla_{e_2}v^{\bot}\right>=0.
\]
We then have,
\begin{eqnarray}
-\left<v,\nabla_v v^{\bot}\right>&=&-\cos\theta\left<v,\nabla_{e_1}v^{\bot}\right>-\sin\theta\left<v,\nabla_{e_2}v^{\bot}\right>\nonumber\\
&=& (\partial/\partial\theta)\left(-\sin\theta\left<v,\nabla_{e_1}v^{\bot}\right>+\cos\theta\left<v,\nabla_{e_2}v^{\bot}\right>\right)\nonumber\\
&=&\left< \partial v/\partial\theta,\nabla_{v^{\bot}} v^{\bot}\right>+\left<v, \nabla_{\partial/\partial\theta}\nabla_{v^{\bot}} v^{\bot}\right>,\nonumber 
\end{eqnarray}
and using (\ref{e:dfff}) we get
\begin{eqnarray}
\left<v,\nabla_v v^{\bot}\right>&=&-\left<v^{\bot},\nabla_{v^{\bot}} v^{\bot}\right>-\left<v, \nabla_{v^{\bot}}\nabla_{\partial/\partial\theta} v^{\bot}\right>\nonumber \\
&=&\left<v, \nabla_{v^{\bot}}v\right>,\nonumber \\
&=&0\label{e:dff}
\end{eqnarray}
Using (\ref{e:dff}), along the contact curve $\phi$, we have:
\[
\left<e_1,\nabla_{e_1}e_2\right>=\left<e_2,\nabla_{e_2}e_1\right>=0,
\]
and therefore, 
\begin{eqnarray}
\left<v,\nabla_{v^{\bot}}v^{\bot}\right>&=&-\cos\theta\left<e_1,\nabla_{e_1}e_2\right>-\sin\theta\left<e_2,\nabla_{e_2}e_1\right>\nonumber\\
&=&0\label{e:importantzero}
\end{eqnarray}
Substituting (\ref{e:importantzero}) into (\ref{e:geodesicequat}) we get
\[
\ddot{\phi}=-\dot{v}^{\bot}=-\phi+(k_1\sin^2\theta+k_2\cos^2\theta)\,N.
\]
Hence $\ddot{\phi}$ lies in the plane $\phi\wedge N$ and thus $\nabla_{\dot{\phi}}\dot{\phi}=0$. Thus the Reeb vector field is the oriented lines
tangent to $S$ that are orthogonal to a geodesic.

This completes the proof of Theorem \ref{t2}. $\;\qquad\Box$

\section{Intersection Tori of Null Hypersurfaces}\label{s4}

Given a smooth convex surface $S\subset {\mathbb R}^3$, the set of oriented outward-pointing normal geodesics forms a surface $\Sigma$ in 
${\mathbb L}({\mathbb R}^3)$
which is Lagrangian and totally real away from umbilic points on $S$ \cite{gag1} \cite{gk1}. 

A normal neighbourhood of $\Sigma$ can be constructed by considering the set
\[
{\mathcal N}_a(\Sigma)=\{\gamma\in{\mathbb L}({\mathbb R}^3)\;|\;\exists\gamma_0\in \Sigma{\mbox{ s.t. }}
\gamma\cap\gamma_0=p\in S{\mbox{ and }} \dot{\gamma}\cdot\dot{\gamma}_0\geq \cos a\;\},
\]
for $a\in[0,\pi/2)$. It is not hard to see that ${\mathcal N}_0(\Sigma)=\Sigma$, while for $a>0$ the 4-manifold ${\mathcal N}_a(\Sigma)$ 
is a disc bundle over $\Sigma$ which is a normal neighbourhood of $\Sigma$ in ${\mathbb L}({\mathbb R}^3)$. 
Moreover the boundary of the normal neighbourhood are the constant angle hypersurfaces introduced in Section \ref{s3}: 
$\partial{\mathcal N}_a(\Sigma)={\mathcal H}_a(S)$. 

Thus for $a=\pi/2$, the null hypersurface ${\mathcal H}_{\pi/2}(S)={\mathcal H}(S)$ that we have been studying are the boundaries of 
the normal neighbourhood of Lagrangian surfaces in ${\mathbb L}({\mathbb R}^3)$.

Consider as a local geometric model, a pair of Lagrangian discs intersecting at an isolated point, given by the
oriented outward-pointing normal lines to two convex spheres $S_1$ and $S_2$, viewed as surfaces  $\Sigma_1$ and $\Sigma_2$ in 
${\mathbb L}({\mathbb R}^3)$.

These intersect in two points $\Sigma_1\cap\Sigma_2=\{\gamma_1,\gamma_2\}$, which when viewed in ${\mathbb R}^3$ are the pair of oriented lines that
attain the following min/max quantities of the two-point distance function:
\[
\min_{p_1\in S_1}\max_{p_2\in S_2}d(p_1,p_2)
\qquad{\mbox{and}}\qquad
\max_{p_1\in S_1}\min_{p_2\in S_2}d(p_1,p_2).
\]

To localize this argument, choose $\gamma_1$ say and consider only discs $D_1\subset S_1$ and $D_2\subset S_2$ about the associated points. Let 
$\Sigma_1$ and $\Sigma_2$ be the oriented normal lines to these discs so that $\Sigma_1\cap\Sigma_2=\{\gamma_1\}$.

About each disc the boundary of a normal neighbourhood as constructed above is ${\mathcal H}(D_j)$ and the intersection 
${\mathcal H}(D_1)\cap{\mathcal H}(D_2)$ is the disjoint union of two tori - each common tangent line to $D_1$ and $D_2$ has two orientations.

For simplicity, let $S$ be a round sphere of radius $r_0$ and centre the origin $(0,0,0)$ in ${\mathbb R}^3$. 
The set of oriented lines normal to $S$ is equal to the set of oriented lines passing through the origin. This is an embedded holomorphic Lagrangian 
sphere $\Sigma\equiv S^2$ given by the zero section of $T{\mathbb S}^2$. In local coordinates this is $\xi\mapsto(\xi,\eta=0)$.

On the other hand, the set of oriented lines tangent to $S$ can be characterized as those oriented lines whose perpendicular distance from the origin is  
$r_0$.  The perpendicular distance to the origin of an oriented line $(\xi,\eta)$ is given by:
\[
\chi=\frac{2|\eta|}{1+\xi\bar{\xi}}.
\]

The hypersurface ${\mathcal H}(S)$ is given locally by 
\[
\xi=\frac{\nu+e^{iA}}{1-\bar{\nu}e^{iA}}
\qquad\qquad
\eta={\textstyle{\frac{1}{2}}}(1+\xi\bar{\xi})r_0e^{iA}=-\frac{(1+\nu\bar{\nu})}{(1-\bar{\nu}e^{iA})^2}r_0e^{iA},
\]
for $(\nu,A)\in{\mathbb C}\times S^1$.

Note that passing to the constant angle hypersurface ${\mathcal H}_a(S)$ here does not change the picture, as the constant angle hypersurface of 
a round sphere with a given radius is the tangent hypersurface of a round sphere with a different radius.

\newpage

\noindent {\bf Proof of Theorem \ref{t3}}:
 
Consider the intersection of the two such tangent hypersurfaces ${\mathcal H}(S_1)$ and 
${\mathcal H}(S_2)$, where $S_1$ and $S_2$ are spheres of radii $r_1\geq r_2$, respectively, whose centres are separated  
by a distance $l$. By a translation and a rotation move the centre of the larger sphere to the origin and the centre of the smaller 
sphere to the positive $x^3$ axis. The Lagrangian sections $\Sigma_1$ and $\Sigma_2$ then intersect at the oriented line along the
$x^3$-axis

The hypersurface ${\mathcal H}(S_1)$ is given by
\[
\frac{2|\eta|}{1+\xi\bar{\xi}}=r_1,
\]
while we can translate ${\mathcal H}(S_2)$ by $l$ to move the centre to the origin which yields a change $\eta\rightarrow\eta+l\xi$ and then it 
is given by
\[
\frac{2|\eta+l\xi|}{1+\xi\bar{\xi}}=r_2.
\]
These are the two equations we must solve to find the intersection.

The first equation is readily solved in polar coordinates 
\[
\xi=Re^{i\theta}
\qquad\qquad 
\eta={\textstyle{\frac{1}{2}}}(1+R^2)r_1e^{i\psi},
\]
for $R\in{\mathbb R}_+$, $\theta\in [0,2\pi)$ and $\psi=\psi(R,\theta)$.

Substituting this into the second equation yields
\[
l^2\frac{R^2}{(1+R^2)^2}+lr_1\cos(\psi-\theta)\frac{R}{1+R^2}+\frac{1}{4}(r_1^2-r_2^2)=0.
\]
The set of solutions to this equation depends upon the relative values of $l$, $r_1$ and $r_2$.  Switching to spherical polar coordinates
by the substitution $R=\tan(\phi/2)$, we can write this as a quadratic equation for $e^{i\psi}$ thus:
\[
lr_1\sin\phi\; e^{2i(\psi-\theta)}+\left(l^2\sin^2\phi+r_1^2-r_2^2\right)e^{i(\psi-\theta)}+lr_1\sin\phi=0.
\]
A solution of this only exists if the discriminant is negative, so that
\[
e^{i\psi}=\left(-K\pm\sqrt{1-K^2}\;i\right)e^{i\theta},
\]
where $K(\phi)$ is the function
\[
K=\frac{r_1^2-r_2^2+l^2\sin^2\phi}{2lr_1\sin\phi}.
\]
Clearly we must have $K^2\leq 1$ for there to be a solution, which implies that
\begin{equation}\label{e:ineq1}
r_1-r_2\leq l\sin\phi\leq r_1+r_2.
\end{equation}

Thus, for a solution to exist we must have $l\geq r_1-r_2$, i.e. one sphere cannot lie completely inside the other sphere. When equality holds, 
the surfaces $S_1$ and $S_2$
intersect at a single point $p$, and the solution set is a circle (parameterized by $\theta$) with $\phi=\pi/2$. This is the circle of oriented lines in 
the common tangent plane $T_pS_1=T_pS_2$ which forms a null curve in ${\mathbb L}({\mathbb R}^3)$.

On the other hand, the surfaces $S_1$ and $S_2$ intersect for  $r_1-r_2<l\leq r_1+r_2$ and the tangent lines (with both orientations) to the intersection
circle are contained in both ${\mathcal H}(S_1)$ and ${\mathcal H}(S_2)$. In fact, ${\mathcal H}(S_1)\cap{\mathcal H}(S_2)=T^2$ in this range and
the intersection set is connected.

Finally, if  $l> r_1+r_2$, the intersection set has two connected components, both tori, which are related by flipping the orientation of the 
common tangent lines.

Let us now compute the induced metric on the solution set when $l>r_1-r_2$, i.e. the intersection tori. The torus is given by
local sections in polar coordinates 
\[
\xi=Re^{i\theta}
\qquad\qquad
\eta={\textstyle{\frac{1}{2}}}(1+R^2)r_1\left(-K\pm\sqrt{1-K^2}\;i\right)e^{i\theta},
\] 
with
\[
K=\frac{(r_1^2-r_2^2)(1+R^2)^2+4l^2R^2}{4lr_1R(1+R^2)}.
\]
This surface has a complex point iff at the point \cite{gk1}
\[
\sigma=-\partial_\xi\bar{\eta}=-{\textstyle{\frac{1}{2}}}e^{-i\theta}\left(\partial_R-\frac{i}{R}\partial_\theta\right)\bar{\eta}=0.
\]
For the torus, a computation shows that 
\[
|\sigma|^2=\frac{r_1^2r_2^2l^2\cos^2\phi}{\left(l^2\sin^2\phi -(r_1-r_2)^2\right)\left((r_1+r_2)^2-l^2\sin^2\phi\right)},
\]
On the other hand, the pullback of the symplectic form (\ref{e:symplectic}) to a section is
\[
\Omega|_\Sigma=\lambda\frac{2id\xi\wedge d\bar{\xi}}{(1+\xi\bar{\xi})^2}
={\mathbb I}{\mbox{m}}\left[\partial_\xi\left(\frac{\eta}{(1+\xi\bar{\xi})^{2}}\right)\right]2id\xi\wedge d\bar{\xi};
\] 
which in our case is
\[
\lambda=-\frac{l(l^2\sin^2\phi -r_1^2-r_2^2)\cos\phi}
  {2[\left(l^2\sin^2\phi -(r_1-r_2)^2\right)\left((r_1+r_2)^2-l^2\sin^2\phi\right)]^{\scriptstyle\frac{1}{2}}}.
\]
Thus the determinant of the metric induced on $T^2$ by the neutral metric (\ref{e:metric}) is \cite{gk1} 
\[
{\mbox{det }}{\mathbb G}|_{T^2}=\lambda^2-|\sigma|^2=-{\textstyle{\frac{1}{4}}}l^2\cos^2\phi,
\] 
If $r_1-r_2<l<r_1+r_2$ then the surfaces $S_1$ and $S_2$ intersect on a circle and the tangent lines to this circle are common to 
${\mathcal H}(S_1)$ an ${\mathcal H}(S_1)$. These lines are horizontal, so $\phi=\pi/2$ along them. Thus, when the surfaces intersect, the
tangent hypersurfaces intersect along a torus that is Lorenz and totally real, except at a  curve of complex points where the induced metric is 
degenerate.

If $l>r_1+r_2$ then the surfaces $S_1$ and $S_2$ do not intersect, $\phi>\pi/2$, and so the tangent hypersurfaces intersect along a pair of
tori (opposite orientations on the same line) that are totally real and Lorentz.

This completes the proof of Theorem \ref{t3}.
\vspace{0.1in}

\section{Neutral Causal Topology}\label{s5}

This section contains a discussion of the preceding constructions with a view to explaining the motivation behind them, to put them in a broader context
and to indicate their possible applications to 4-manifold topology. 

Neutral metrics offer us geometric tools that are
sensitive to the underlying topology, at the level of the metric, rather than its curvature. 
This can be seen from a geometric, analytic and topological perspective. While the individual scenarios
are classical in some sense, it is their concatenation that is of particular interest. 

As the discussion necessitates spanning a number of areas, the bibliography will be selective rather than exhaustive. Further aspects of neutral metrics
which we do not discuss can be found foe example in \cite{BaGaGaV} \cite{DaGaM} \cite{law1} \cite{law3} and references therein. 

A fundamental observation is that, point-wise, the null cone of a neutral metric is a cone over a torus. Since the cross-section is not simply 
connected, under the right circumstances, it is possible to encode topological information in the null cone of a neutral metric.  

Put another way, 
the metric must fit with the underlying 4-manifold topology and so, for example, there are obstructions to their existence. 
For compact smooth 4-manifolds the matter is clarified by the following theorem, which uses Hirzebruch and Hopf's 1950's work on plane fields
\cite{HaH}:

\vspace{0.1in}
\begin{Thm}\cite{kam02} \cite{Mats91}
Let ${\mathbb N}^4$ be a smooth compact 4-manifold admitting a neutral metric. Then
\[
\chi({\mathbb N}^4)+\tau({\mathbb N}^4)=0{\mbox{ mod }}4
\qquad{\mbox and }\qquad
\chi({\mathbb N}^4)-\tau({\mathbb N}^4)=0{\mbox{ mod }}4,
\]
where $\chi({\mathbb N}^4)$ is the Euler number and $\tau({\mathbb N}^4)$ the Hirzebruch signature of ${\mathbb N}^4$.

Moreover, if ${\mathbb N}^4$ is simply connected, these conditions are sufficient for the existence of a neutral metric.
\end{Thm}
\vspace{0.1in}

Thus, neither ${\mathbb S}^4$ nor $C{\mathbb P}^2$ admit a neutral metric, while K3 manifolds do. If one 
demands further that the neutral metric is K\"ahler with respect to some compatible complex structure, then the list of compact manifolds becomes 
smaller \cite{pet}. Thus, a K3 manifold admits a neutral metric, but not a neutral K\"ahler metric.

One motivation for this paper is to consider the extension of the above to 4-manifolds with boundary and to ask: what does the null boundary 
geometry see of the interior of a neutral 4-manifold? 

Similar to the holographic principle, but predating it by 60 years, the X-ray transform, or
strictly speaking its symmetric reduction to the Radon transform, is used every day in hospitals' CAT-scans and achieves this feat at the
level of functions \cite{DaW}. 

That is, given a real-valued function on ${\mathbb L}({\mathbb R}^3)$ (the difference between intensities along a ray, or oriented geodesic) 
reconstruct a function on ${\mathbb R}^3$ (the material density). The compatibility requirement is that the function on 
${\mathbb L}({\mathbb R}^3)$ satisfy the flat ultra-hyperbolic equation \cite{john}. 

Hilbert and Courant showed that the appropriate cauchy hypersurface for the ultra-hyperbolic equation is null - as otherwise there are consistency 
conditions on the cauchy data in the initial value formulation \cite{CaH}. Thus we encounter our first evidence, from analysis, that null boundaries 
are natural for neutral metrics. 

The nullity of the boundary introduces more structure than in the Riemannian case, in particular, a null foliation. Moreover, the neutral
metric has more structure again than a Lorentz metric with null boundary, namely, two distributions of totally null planes. It is to this
structure that we look for echoes of the interior geometry.

It also follows from Hirzebruch and Hopf that all open 4-manifolds admit a neutral metric, and so the question arises about the compactification
of such neutral 4-manifolds. Since the null cone is preserved by conformal transformations, it is natural then to look at the conformal compactification
of open neutral 4-manifolds. In Section \ref{s2} we investigated the simplest case, neutral flat 4-space. 

Note that in this paper, the conformal compactification has non-empty boundary, in contrast to earlier consideration of neutral conformal 
compactifications into a 4-manifold without boundary \cite{wood}.

Conformal compactifications of both Lorentz and Riemannian cases have been considered in some detail (e.g. \cite{and01} \cite{PaR}). 
In the Riemannian case it is natural to assume that the gradient of the conformal factor is nowhere zero on the boundary, while in the Lorentzian case it 
vanishes at points (at $i^0,i^\pm $ \cite{PaR}). In the neutral case we consider, it vanishes along a link in the boundary (property (iv) in 
Theorem \ref{t1}) 
and this is the manner in which the geometric topology intervenes.

For a flat 4-dimensional universe with two times,  the spacelike and timelike infinities are Hopf linked in the boundary. This
is the simplest situation and one would expect these linked infinities to also link to whatever topology the boundary has. 

What's more, the boundary is foliated by Lorentz tori about the link. This feature persists for other neutral conformal compactifications and is amenable 
to surgery along the link in a manner that preserves the null cone structure. 

Indeed, it should be possible to do surgery at infinity and preserve not just the conformal structure, but certain curvature conditions. What precisely 
the conditions are depends on the amount of flexibility required. We can impose a stiff restriction such as K\"ahler or a softer one such as 
anti-self dual, scalar flat.

Certainly all of the examples considered in this paper are conformally flat and 
scalar flat, and this is a natural class in which to do the surgery. In fact, by 2-handle attachments to the 4-ball one can generate the 
conformal compactifications of all of the oriented geodesic spaces ${\mathbb L}({\mathbb M}^3)$. We postpone the details of this aspect to a later paper. 

The fact that the $\alpha$-planes and $\beta$-planes are integrable in the boundary gives a sense in which the conformal compactification is 
asymptotically well-behaved. This can be traced back to the fact that ${\mathbb R}^{2,2}$ is 2-connected at infinity and therefore the neutral
metric has nothing to hang on to at infinity.  
This observation suggests the use of neutral metrics to detect topology at infinity via their conformal compactifications. 

This example represents the 0-handle in a neutral version of Kirby calculus \cite{gompf} \cite{kirby} and therefore
acts as a basis for handle-body constructions. The degenerate Lorentz structure on the boundary has preferred curves along which to attach 2-handles 
and the Lorentz tori give framings in the right circumstances.

In contrast, in the tangent hypersurfaces of Section \ref{s3} the $\alpha$-planes and $\beta$-planes were found to be contact. The boundary is not a 
3-sphere but a circle bundle with Euler number 2 over a 2-sphere. The fibres are null and the totally null planes rotate around the fibre as one 
traverses a fibre. 

The neutral 4-manifold bounded by the tangent hypersurface and its generalisation, the constant angle hypersurface, were introduced in \cite{gk2} 
to prove a global version of a classical result of Joachimsthal.

The study of Legendrian knots and their invariants is a well-established area in symplectic topology \cite{EaH01} \cite{EaH03}. Generally the knots
are in the 3-sphere, but many results extend to more general contact 3-manifolds. In Section \ref{s3}, we have a non-simply connected 3-manifold
with a pair of independent contact structures. 

A curve ${\mathcal C}$  on the null boundary ${\mathcal H}(S)$  corresponds to an oriented tangent line field along a curve $c$ in $S$. 
The classical Thurston-Bennequin index and the rotation index of the curve ${\mathcal C}$ can be expressed in terms of the twisting of the 
oriented line and the rotation of $c$ in $S$ through the neutral structure. 

The fact that the Reeb vector field is the set of normal lines to the geodesics on the surface is important. This means that {\it Reeb chords} - 
Reeb flow-lines that begin and end on a Legendrian knot,  minimize the induced two point distance function on the knot. Reeb chords play a 
critical role in knot contact homology \cite{EaEaNaS} as they represent crossings in the Legendrian projection. Further details of these neutral knot 
invariants will appear in a future paper.
 
Many peculiarities of 4-dimensional manifolds (as distinct from higher dimensions) arise because generic 2-discs are only immersed rather than embedded 
and one loses the ability to contract loops across such discs. Thus, attempts to exploit assumptions of simple connected-ness become more difficult and 
higher dimensional techniques fail.

The local model of these double points and their normal neighbourhood play key roles in our understanding of 4-manifolds (or lack thereof). In particular, 
the intersection of the boundaries of the normal neighbourhoods are tori, called {\it distinguished} in \cite{kirby} and {\it characteristic} 
in \cite{GaM}. It is this basic model we set out to find a neutral geometric interpretation for in Section \ref{s4}.

In the first instance, the intersecting discs should be flexible enough to be pushed around and stretched, for example in Casson's famous ``finger-move'' 
\cite{GaM}. In the geometric category, Lagrangian discs are certainly flexible enough for this task since they satisfy the h-principle \cite{EaM}. 

The boundary of a normal neighbourhood of these Lagrangian discs can be identified with the tangent hypersurface introduced in Section \ref{s3}.
The work of Casson in the 1970's involves repeated attempts to remove unwanted intersections by adding thickened discs that remove the double point.
The issue, peculiar to dimension 4, is that such discs may themselves have double points, leading to an iterative chain of operations seeking to push
the double point out to infinity.

While Casson achieved this at a homotopic level, giving rise to {\it flexible handles}, it certainly fails in the smooth category due to implications
of the work of Donaldson \cite{DaK}. A motivation for the present work is to explore this gap by geometrizing the boundary with a neutral metric 
and carrying it along in this iterative construction.

In Section \ref{s4} we found a geometric model for the intersection torus of a double point. Similar natural neutral constructions exist for the
other elements of the Casson handle, such as the Whitehead double, although in the tangent model a twisted version is more natural. Further details
of these constructions will be given in a future paper.

\end{document}